\documentclass[11pt]{article}
\usepackage{amsfonts}
\usepackage{amstext}
\usepackage{amssymb}
\usepackage{amsmath}
\usepackage{amsthm}
\usepackage{amscd}

\usepackage{euscript}
\usepackage{mathrsfs}
\usepackage{indentfirst}
\usepackage{enumerate}
\usepackage{leftidx}
\usepackage{hyperref}

\usepackage{xcolor}

\newtheorem{thm}{Theorem}[section]
\newtheorem{prop}[thm]{Proposition}
\newtheorem{lem}[thm]{Lemma}
\newtheorem{cor}[thm]{Corollary}

\theoremstyle{definition}
\newtheorem{defn}[thm]{Definition}

\theoremstyle{remark}

\numberwithin{equation}{section}

\DeclareMathOperator{\tr}{Tr}
\DeclareMathOperator{\SF}{sf}
\DeclareMathOperator{\ind}{Ind}

\DeclareMathOperator{\ch}{ch}
\DeclareMathOperator{\cok}{coker}

\DeclareMathOperator{\End}{End}

\DeclareMathOperator{\SL}{SL}
\DeclareMathOperator{\SO}{SO}
\DeclareMathOperator{\Spin}{Spin}

\newcommand{\CC}{\mathbb{C}}
\newcommand{\be}{\begin{equation}}
\newcommand{\ee}{\end{equation}}

\title{On the Witten Rigidity Theorem for \\ Odd Dimensional Manifolds}

\author{Fei HAN\footnote{Department of Mathematics, National University of Singapore,
Block S17, 10 Lower Kent Ridge Road, Singapore 119076. (mathanf@nus.edu.sg)},
\
Jianqing YU\footnote{School of Mathematical Sciences,
University of Science and Technology of China,
96 Jinzhai Road, Hefei, Anhui 230026,
P. R. China. (jianqing@ustc.edu.cn)}}

\date{}

\begin{document}

\maketitle

\begin{abstract}
We establish several Witten type rigidity and vanishing theorems for twisted Toeplitz operators on
odd dimensional manifolds. We obtain our results by combining the modular method, modular transgression and some careful analysis of odd Chern classes for cocycles in odd $K$-theory. Moreover we discover that in odd dimensions, the fundamental group of manifolds plays an important role in the rigidity.
\end{abstract}

\tableofcontents

\section{Introduction}\label{sec1}

Let $M$ be a closed smooth manifold and
$P$ be a Fredholm operator on $M$.
We assume that a compact connected Lie group $G$ acts on $M$ nontrivially and
that $P$ is $G$-equivariant, by which we mean it commutes with the $G$ action.
Then the kernel and cokernel of $P$ are finite dimensional representations of $G$.
The equivariant index of $P$ is the virtual character of $G$ defined by
\begin{equation}
\ind(h,P)=\tr\big[h\big|_{\ker P}\big]-\tr\big[h\big|_{\cok P}\big],\quad h\in G.
\end{equation}

$P$ is said to be \emph{rigid} for this $G$ action if $\ind(h,P)$ does not
depend on $h\in G$. Furthermore, we say that $P$ has \emph{vanishing property} if $\ind(h,P)$ is identically zero.
To study rigidity and vanishing properties of Fredholm operators,
we only need to restrict to the case when $G=S^1$.

It is well known that classical operators: the signature operator for oriented manifolds, the
Dolbeault operator for almost complex manifolds and the Dirac operator for spin manifolds are
rigid \cite{MR0278334}. In \cite{MR970288}, Witten derived a series of
twisted Dirac operators on the free loop space $LM$ of a spin manifold $M$.
The elliptic genus constructed by Landweber-Stong \cite{MR948178}
and Ochanine \cite{MR895567} in a topological way
turns out to be the index of one of these operators.
Motivated by physics, Witten conjectured that these elliptic operators should be rigid.
In particular, as a highly nontrivial consequence, the twisted operator
\begin{equation*}
D\otimes T_\CC M
\end{equation*}
which is known as the {\it Rarita-Schwinger operator} \cite{MR885560} is rigid. We
refer to \cite{MR970281} for a brief early history of the subject.

The Witten conjecture was first proved by Taubes \cite{MR998662} and Bott-Taubes \cite{MR954493}.
Hirzebruch \cite{MR981372} and Krichever \cite{MR1048541} proved Witten's conjecture for almost complex manifold case.
In \cite{MR1331972,MR1396769}, using the modular invariance property, Liu
presented a simple and unified proof as well as various vast generalizations of the Witten conjecture.
Several new vanishing theorems were also established in \cite{MR1331972,MR1396769}. In \cite{MR1722036}, Dessai established the rigidity and vanishing theorems for spin$^c$ case.
Liu-Ma \cite{MR1756105,MR1969037} and Liu-Ma-Zhang \cite{MR1870666,MR2016198}
generalized the rigidity and vanishing theorems to the family case
 on the levels of equivariant Chern character and of equivariant $K$-theory.
However, since Dirac operators on odd dimensional manifolds are self ajoint and therefore have index zero,
the rigidity and vanishing properties for twisted Dirac operators
make sense only for even dimensional manifolds.

Now let $M$ be an odd dimensional closed smooth spin Riemannian manifold.
The appropriate index to consider on $M$ is that of twisted Toeplitz operators (\cite{MR679698,MR1231957}),
which gives the natural pairing between odd $K$-homology and odd $K$-group.
Thus it fits with the interpretation of the index of twisted Dirac operator
on even dimensional manifolds as a pairing between even  $K$-homology and even $K$-group.
An element of $K^{-1}(M)$ can be represented by a differentiable map from $M$ into the general linear group
\begin{equation*}
g:M\longrightarrow {\rm GL}(N,\mathbb{C}),
\end{equation*}
where $N$ is a positive integer. We recall the definition of Toeplitz operators as follows.

Let $\Delta(TM)$ be the Hermitian bundle of spinors and $\mathcal{E}$ be a complex
Hermitian vector bundle carrying a Hermitian connection $\nabla^{\mathcal{E}}$ over $M$.
The twisted Dirac operator $D\otimes \mathcal{E}$
induces a splitting of $L^2(M,\Delta(TM)\otimes \mathcal{E})$, the $L^2$-completion of
the space $\Gamma(M,\Delta(TM)\otimes \mathcal{E})$ of smooth sections of $\Delta(TM)\otimes \mathcal{E}$ over $M$, into an
orthogonal direct sum as
\begin{equation*}
L^2(M,\Delta(TM)\otimes \mathcal{E})=\bigoplus_{\lambda\in\,{\rm Spec}\,(D\otimes \mathcal{E})} \mathcal{E}_{\lambda}\ ,
\end{equation*}
where $\mathcal{E}_{\lambda}$ is the eigenspace associated to the eigenvalue $\lambda$ of $D\otimes \mathcal{E}$. Set
\begin{equation*}
L_+^2(M,\Delta(TM)\otimes \mathcal{E})=\bigoplus_{\lambda\geq 0}\mathcal{E}_{\lambda},
\end{equation*}
and denote by $P_+$ the orthogonal projection from $L^2(M,\Delta(TM)\otimes \mathcal{E})$ to $L_+^2(M,\Delta(TM)\otimes \mathcal{E})$.

Now consider the trivial vector bundle $\mathbb{C}^N|_M$ of rank $N$ over $M$. We equip $\mathbb{C}^N|_M$ with the canonical
trivial metric and connection.
Then $P_+$ extends naturally to an orthogonal projection
\begin{equation*}
L^2(M,\Delta(TM)\otimes \mathcal{E}\otimes \mathbb{C}^N|_{M})\longrightarrow L_+^2(M,\Delta(TM)\otimes \mathcal{E}\otimes \mathbb{C}^N|_{M})
\end{equation*}
by acting as identity on $\mathbb{C}^N|_{M}$.
We still denote this extension by $P_+$.

The map $g$ can be interpreted as an automorphism of $\mathbb{C}^N|_M$. Moreover,
$g$ extends naturally to a bounded map from
\begin{equation*}
L^2(M,\Delta(TM)\otimes \mathcal{E}\otimes \mathbb{C}^N|_{M})
\end{equation*}
to itself by acting as identity on $L^2(M,\Delta(TM)\otimes \mathcal{E})$. We still denote this extended map by $g$.

With the above data given, the twisted Toeplitz operator
associated to $D\otimes \mathcal{E}$ and $g$ can be defined as
\begin{align}
&\hspace{2em}\mathcal{T}\otimes \mathcal{E}\otimes (\mathbb{C}^N|_{M},g):=P_+gP_+:
\\
&L_+^2\big(M,\Delta(TM)\otimes \mathcal{E}\otimes \mathbb{C}^N|_{M}\big)
\longrightarrow L_+^2\big(M,\Delta(TM)\otimes \mathcal{E}\otimes \mathbb{C}^N|_{M}\big).\notag
\end{align}

The important fact is that $\mathcal{T}\otimes \mathcal{E}\otimes (\mathbb{C}^N|_{M},g)$ is a Fredholm operator.
Moreover, we can compute its index as follows (see \cite{MR679698,MR1231957}),
\begin{equation}
\begin{split}\label{35}
&\ind\big(\mathcal{T}\otimes \mathcal{E}\otimes (\mathbb{C}^N|_{M},g)\big)
\\
&=-\Big\langle
\widehat{A}(TM,\nabla^{TM})
\ch(\mathcal{E},\nabla^{\mathcal{E}})
\ch(\mathbb{C}^N|_M,g,d),\,[M]\Big\rangle,
\end{split}
\end{equation}
where $[M]$ denotes the fundamental class of $M$. See Section \ref{sec2.1} for the relevant characteristic forms.

In \cite{MR2532715}, Liu-Wang for the first time study the rigidity and vanishing properties of Toeplitz operators.
They obtained a very interesting result \cite[Theorem 2.4]{MR2532715} which states that
$$\mathcal{T}\otimes (\mathbb{C}^N|_{M},g)$$
is rigid and this can be thought of as an odd analogue of Atiyah-Hirzebruch theorem \cite{MR0278334}. Furthermore,
they established the rigidity for twisted Toeplitz operators associated to the Witten bundles by Liu's approach \cite{MR1331972,MR1396769} under the assumption that the fixed point sets of the group action are $1$-dimensional.

The purpose of the present paper is to extend their rigidity and vanishing properties
for twisted Toeplitz operators to the cases of fixed points of general dimensions. Let
\begin{align} \Theta_{2}(TM)&=\bigotimes_{n=1}^\infty S_{q^n}(T_\mathbb{C}M)\otimes
\bigotimes_{n=1}^{\infty}\Lambda_{-q^{n-{1/2}}}(T_\CC M),
\\
\Theta_{3}(TM)&=\bigotimes_{n=1}^\infty S_{q^n}(T_\mathbb{C}M)\otimes
\bigotimes_{n=1}^{\infty}\Lambda_{q^{n-{1/2}}}(T_\CC M).\end{align}
We obtain the following result (see the proof after Corollary \ref{V=TM}).
\begin{thm}\label{main}
Let $M$ be an odd dimensional smooth closed spin Riemannian manifold with a nontrivial circle action.
Let $g:M\longrightarrow \SO(N)$ be an $S^1$-invariant cocycle in the odd real $K$-theory of $M$.
Suppose $M$ is simply connected and $H^3(M, \mathbb{R})=0$.  Then the {\it Toeplitz-Witten operators}
\begin{equation*}
\mathcal{T}\otimes \Theta_2(TM)\otimes (\mathbb{C}^N|_{M},g),\quad
\mathcal{T}\otimes \Theta_3(TM)\otimes (\mathbb{C}^N|_{M},g)
\end{equation*}
are rigid. In particular, the {\it Toeplitz-Rarita-Schwinger operator}
$$\mathcal{T}\otimes T_\CC M \otimes (\mathbb{C}^N|_{M},g) $$
is rigid.
\end{thm}

We would like to point out that if we extend $D\otimes \mathcal{E}$ as an operator from
\begin{equation*}
\Gamma(M,\Delta(TM)\otimes \mathcal{E}\otimes \mathbb{C}^N|_{M})
\end{equation*}
to itself by acting as identity on $\mathbb{C}^N|_{M}$,
the equivariant index
\begin{equation*}
\ind(h, \mathcal{T}\otimes \mathcal{E}\otimes (\mathbb{C}^N|_{M},g)), \quad h\in S^1
\end{equation*}
computes the {\it equivariant spectral flow} (see \cite[Theorem 2.7]{MR2200374})
\begin{equation*}
\SF(h, D\otimes \mathcal{E}, g(D\otimes \mathcal{E})g^{-1})
\end{equation*}
for the path
\begin{equation*}
(1-u)D\otimes \mathcal{E}+ug(D\otimes \mathcal{E})g^{-1},\quad u\in [0,1].
\end{equation*}
Hence when $M$ is simply connected and $H^3(M, \mathbb{R})=0$, the equivariant spectral flows
\begin{equation*}
\SF\Big(h, D\otimes \Theta_2(TM), g(D\otimes \Theta_2(TM))g^{-1}\Big)
\end{equation*}
and
\begin{equation*}
\SF\Big(h, D\otimes \Theta_3(TM), g(D\otimes \Theta_3(TM))g^{-1}\Big)
\end{equation*}
do not depend on $h\in S^1$. In particular, the equivariant spectral flow
\begin{equation*}
\SF\Big(h, D\otimes T_\mathbb{C}M, g(D\otimes T_\mathbb{C}M)g^{-1}\Big)
\end{equation*}
for the Rarita-Schwinger operators does not depend on $h\in S^1$.

Actually we have obtained more general results; see Theorems \ref{24}, \ref{37} and Corollaries \ref{V=TM}, \ref{V=0}.
Our approach is to combine Liu's modularity methods \cite{MR1331972,MR1396769},
the modular transgression in \cite{MR2495834} and some careful analysis of Chern classes for cocycles in odd $K$-theory.
Also, parallel to \cite[Theorem 2.4]{MR2532715}, we are able to show
the rigidity of the {\it Toeplitz-Signature operator}
\begin{equation*}
\mathcal{T}\otimes \Delta(TM)\otimes (\mathbb{C}^N|_{M},g)
\end{equation*}
in Theorem \ref{36} without assuming the simply connectedness of $M$ and the vanishing of
of $H^3(M,\mathbb{R})$.
Our results should have applications to
the study of topology of odd dimensional manifolds.

A brief outline of the paper is as follows. In Section \ref{sec2}, we first review some important characteristic
forms and then study modularities of odd Chern character forms. We then introduce elliptic
genera and Witten genera for the pair $(M, [g])$ on odd dimensional manifolds as well as relate
them to indices of Toeplitz operators. Section \ref{sec3} is devoted to the study of
rigidity of the twisted Toeplitz operators.

\vspace{3mm}\textbf{Acknowledgements}\
The authors would like to thank Professor Weiping Zhang for helpful suggestions.
F. H. was partially supported by the grant AcRF R-146-000-163-112 from National University of Singapore.
J. Y. was supported by the China Postdoctoral Science Foundation 2014M551805 and NSFC 11401552.
Part of work was done when the second author was visiting the Max-Planck-Institut f\"{u}r Mathematik at Bonn.

\section{Elliptic genera on odd dimensional manifolds}\label{sec2}

\subsection{Some Characteristic Forms}\label{sec2.1}
Let $M$ be a closed smooth Riemannian manifold. Let $\nabla^{TM}$ be the associated Levi-Civita
connection on $TM$ and $R^{TM}=(\nabla^{TM})^2$ be the curvature of $\nabla^{TM}$. Let $\widehat{A}(TM,\nabla^{TM})$
and $\widehat{L}(TM,\nabla^{TM})$ be the characteristic forms defined, respectively, by (c.f. \cite[Section 1.6]{MR1864735},
\cite[Section III.11]{MR1031992})
\begin{equation}
\begin{split}
\widehat{A}(TM,\nabla^{TM})
&={\det}^{1/2}\Big(\frac{\sqrt{-1}R^{TM}/(4\pi)}{\sinh\big(\sqrt{-1}R^{TM}/(4\pi)\big)}\Big),
\\
\widehat{L}(TM,\nabla^{TM})
&={\det}^{1/2}\Big(\frac{\sqrt{-1}R^{TM}/(4\pi)}{\tanh\big(\sqrt{-1}R^{TM}/(4\pi)\big)}\Big).
\end{split}
\end{equation}

Let $W$, $W'$ be two complex vector bundles over $M$ carrying connections $\nabla^{W}$, $\nabla^{W'}$, respectively.
Then the formal difference $W-W'$ carries a naturally induced connection which we denote by $\nabla^{\ominus}$.
We recall that the Chern character form associated to $(W-W',\nabla^{\ominus})$ is (see \cite[Section 1.6]{MR1864735} )
\begin{equation}
\ch(W-W',\nabla^{\ominus})=\tr\Big[\exp\Big(\frac{\sqrt{-1}}{2\pi}R^{W}\Big)\Big]
-\tr\Big[\exp\Big(\frac{\sqrt{-1}}{2\pi}R^{W'}\Big)\Big],
\end{equation}
where $R^{W}$ and $R^{W'}$ denote the curvatures of $\nabla^{W}$ and $\nabla^{W'}$ respectively.

For any $t\in \mathbb{C}$, let
\begin{equation}
\Lambda_t(W)=\sum_{i=0}^{\infty} t^i\Lambda^i(W),\quad
S_t(W)=\sum_{i=0}^{\infty} t^i S^i(W)
\end{equation}
denote, respectively, the total exterior and symmetric powers of $W$, which live in $K(M)[[t]]$.
The following relations between these two operations hold (see \cite[Chapter 3]{MR0224083}),
\begin{equation}
S_t(W)=\frac{1}{\Lambda_{-t}(W)},\quad \Lambda_{t}(W-W')=\frac{\Lambda_t(W)}{\Lambda_t(W')}.
\end{equation}

For a real Euclidean vector bundle $V$ over $M$, we denote by $V_\mathbb{C}$
the complexification of $V$, which carries a naturally induced Hermitian metric.
Set
\begin{equation*}
\widetilde{V_\mathbb{C}}=V_\mathbb{C}-\mathbb{C}^{\dim V} \in K(M).
\end{equation*}

If $V$ carries a spin structure, we denote by $\Delta(V)$ the Hermitian
bundle of spinors associated to $V$.

Let $q=e^{2\pi\sqrt{-1}\tau}$ with $\tau\in \mathbb{H}$, the
upper half complex plane. Set
\begin{equation}
\begin{split}\label{50}
Q_1(V)_{v}&=\Delta(V)\otimes\bigotimes_{n=1}^{\infty}\Lambda_{q^n}(\widetilde{V_\mathbb{C}})\ ,
\\
Q_2(V)_v&=\bigotimes_{n=1}^{\infty}\Lambda_{-q^{n-1/2}}(\widetilde{V_\mathbb{C}})\ ,
\quad\
Q_3(V)_{v}=\bigotimes_{n=1}^{\infty}\Lambda_{q^{n-1/2}}(\widetilde{V_\mathbb{C}})\ .
\end{split}
\end{equation}
Let $\nabla^{V}$ be an Euclidean connection on $V$, which canonically induces
Hermitian connections on the coefficients of the formal Fourier expansions of
$Q_j(V)_{v}$, $j=1,2,3$, in $q^{1/2}$. Thus we get induced connections $\nabla^{Q_j(V)_v}$
with $q^{1/2}$-coefficients on $Q_j(V)_{v}$, $j=1,2,3$.

To express the Chern character forms of $(Q_j(V)_v,\nabla^{Q_j(V)_v})$ explicitly, we recall
the four Jacobi theta functions as follows (see \cite{MR808396}),
\begin{align}
\theta(v,\tau)&=2c(q)q^{1/8}\sin(\pi v)\prod_{n=1}^\infty(1-q^ne^{2\pi\sqrt{-1}v})\prod_{n=1}^\infty(1-q^ne^{-2\pi\sqrt{-1}v}),
\\
\theta_1(v,\tau)&=2c(q)q^{1/8}\cos(\pi v)\prod_{n=1}^\infty(1+q^ne^{2\pi\sqrt{-1}v})\prod_{n=1}^\infty(1+q^ne^{-2\pi\sqrt{-1}v}),
\\
\theta_2(v,\tau)&=c(q)\prod_{n=1}^\infty(1-q^{n-1/2}e^{2\pi\sqrt{-1}v})\prod_{n=1}^\infty(1-q^{n-1/2}e^{-2\pi\sqrt{-1}v}),
\\
\theta_3(v,\tau)&=c(q)\prod_{n=1}^\infty(1+q^{n-1/2}e^{2\pi\sqrt{-1}v})\prod_{n=1}^\infty(1+q^{n-1/2}e^{-2\pi\sqrt{-1}v}),
\end{align}
where $c(q)=\prod_{n=1}^\infty(1-q^{n})$.
By using the Chern roots algorithm as in  \cite[Section 3]{MR2495834}, we obtain the following formulas,
\begin{equation}\label{3}
\begin{split}
\ch(Q_1(V)_v,\nabla^{Q_1(V)_v},\tau)&={\det}^{1/2}\Big(\frac{2\,\theta_1(R^V/(4\pi^2),\tau)}{\theta_1(0,\tau)}\Big),
\\
\ch(Q_j(V)_v,\nabla^{Q_j(V)_v},\tau)&={\det}^{1/2}\Big(\frac{\theta_j(R^V/(4\pi^2),\tau)}{\theta_j(0,\tau)}\Big),\quad j=2,3,
\end{split}
\end{equation}
where $R^{V}$ denotes the curvature of $\nabla^{V}$.

We recall characteristic forms for odd $K$-theory (\cite{MR1231957}, cf. \cite{MR1864735}). Let
\begin{equation*}
g:M\longrightarrow {\rm GL}(N,\mathbb{C})
\end{equation*}
be a smooth map from $M$ to the general linear
group ${\rm GL}(N,\mathbb{C})$ with $N$ a positive integer, and let $\mathbb{C}^N|_M$ denote the
trivial complex vector bundle of rank $N$ over $M$. Then $g$ can be viewed as an automorphism of $\mathbb{C}^N|_M$.

Let $d$ denote a trivial connection on $\mathbb{C}^N|_M$, and we associate with $g$ a natural element
\begin{equation*}
g^{-1}dg=g^{-1}\cdot d\cdot g-d \in \Omega^1(M,\End(\mathbb{C}^N|_M)).
\end{equation*}
Then $\tr\big[(g^{-1}dg)^n\big]$ is closed
for any positive odd integer $n$. Moreover, the cohomology class determined by $\tr\big[(g^{-1}dg)^n\big]$
in $H^{\rm odd}(M,\mathbb{C})$ depends only on the homotopy class $[g]$ of $g$.
We will denote by $c_n(M,[g])$ the cohomology class associated to the closed $n$-form
\begin{equation}
c_{n}(\mathbb{C}^N|_M,g,d)=\Big(\frac{1}{2\pi\sqrt{-1}}\Big)^{(n+1)/2}\tr\big[(g^{-1}dg)^n\big].
\end{equation}

The odd Chern character form $\ch(\mathbb{C}^N|_M,g,d)$ associated to $g$ and $d$ by definition is
\begin{equation}
\ch(\mathbb{C}^N|_M,g,d)=\sum_{n=1}^\infty \frac{n!}{(2n+1)!}c_{2n+1}(\mathbb{C}^N|_M,g,d).
\end{equation}
Alternatively, $\ch(\mathbb{C}^N|_M,g,d)$ is exactly the Chern-Simons form associated to the curve
\begin{equation}\label{53}
\nabla_u=(1-u)d+u\,g^{-1}\cdot d\cdot g= d+ u\,g^{-1}dg, \quad u\in [0,1]
\end{equation}
of connections on $\mathbb{C}^N|_M$, which is such that (see \cite[Section 1]{MR1231957})
\begin{equation}
d\ch(\mathbb{C}^N|_M,g,d)=\ch(\mathbb{C}^N|_M,d)-\ch(\mathbb{C}^N|_M,g^{-1}\cdot d\cdot g).
\end{equation}

\subsection{Modularities of odd Chern character forms}

Let $g:M\rightarrow \SO(N)$ be a smooth map from $M$ to the special orthogonal
group $\SO(N)$ with $N$ a positive integer.
We assume that $N$ is even and large enough. Otherwise, we replace $g$ by
$\begin{pmatrix}
g & 0
\\
0 & I
\end{pmatrix}$ with $I$ a certain identity matrix of large rank.

Let $E$ denote the trivial real vector bundle of rank $N$ over $M$.
We equip $E$ with the canonical trivial metric and trivial connection $d$ which is clearly Euclidean. Set (cf. \eqref{53})
\begin{equation}\label{5}
\nabla_u= d+ u\,g^{-1}dg, \quad u\in [0,1],
\end{equation}
then $\nabla_u$, $u\in [0,1]$ defines a curve of Euclidean connections on $E$ with $\nabla_0=d$
and $\nabla_1=g^{-1}\cdot d\cdot g$. Let $R_u$ be the curvature of $\nabla_u$, then
\begin{equation}\label{4}
R_u=(u^2-u)(g^{-1}dg)^2,\quad u\in [0,1].
\end{equation}

By complexification, the metric and the trivial connection on $E$ induce naturally a Hermitian metric and
a trivial Hermitian connection on $E_\mathbb{C}$.
Also, $g$ extends to a unitary automorphism of $E_\mathbb{C}$,
due to the obvious embedding $\SO(N)\subset{\rm U}(N)$, the unitary group.
Hence, \eqref{5} extends naturally to Hermitian connections on $E_\mathbb{C}$
with curvatures still given by \eqref{4}.

Let $\Delta(E)$ be the spinor bundle of $E$, which is a trivial Hermitian bundle of rank $2^{N/2}$.

Let $\pi_1(\cdot)$ denote the fundamental group of a topological space. We assume that the induced map
\begin{equation*}
g_*:\pi_1(M)\longrightarrow \pi_1(\SO(N))=\mathbb{Z}_2
\end{equation*}
is trivial. Then by \cite[Chapter 1]{MR1867354}, there exists a unique lift (depending on the choice of the base point) to the Spin group $\Spin(N)$,
\begin{equation}
{g}^{\Delta}: M\longrightarrow \Spin(N).
\end{equation}
${g}^{\Delta}$ can be viewed as an automorphism of $\Delta(E)$ preserving the Hermitian metric.
If we lift $d$ on $E$ to be a trivial Hermitian connection $d^\Delta$ on $\Delta(E)$, then
\begin{equation}
\nabla^\Delta_u=(1-u)\,d^\Delta+u\,(g^\Delta)^{-1}\cdot d^\Delta\cdot g^\Delta, \quad u\in [0,1]
\end{equation}
lift the connections in \eqref{5} on $E$ to $\Delta(E)$.

Let $Q_j(E)_{v}$, $j=1,2,3$ be the virtual bundles defined as in \eqref{50}. Then the action of $g$ on $E$
naturally induces actions $g^{Q_j(E)_v}$ on $Q_j(E)_v$. As explained below \eqref{50},
let $\nabla^{Q_j(E)_v}_0$ and $\nabla^{Q_j(E)_v}_1$ denote, respectively,
the connections on $Q_j(E)_v$ induced by $\nabla_0$ and $\nabla_1$.
From \eqref{3} and \cite[Theorem 2.2]{MR2495834}, we get that
\begin{align}
&\ch(Q_1(E)_v,\nabla^{Q_1(E)_v}_0,\tau)-\ch(Q_1(E)_v,\nabla^{Q_1(E)_v}_1,\tau)
\notag
\\
&=-d\int_{0}^{1}\frac{1}{8\pi^2}{\det}^{1/2}\Big(\frac{2\,\theta_1(R_u/(4\pi^2),\tau)}{\theta_1(0,\tau)}\Big)\tr\Big[g^{-1}dg
\frac{\theta_j'(R_u/(4\pi^2),\tau)}{\theta_j(R_u/(4\pi^2),\tau)}\Big]du\notag
\end{align}
and that for $j=2,3$ (compare with \cite[(5.4)-(5.5)]{MR2495834}),
\begin{align}
&\ch(Q_j(E),\nabla^{Q_j(E)}_0,\tau)-\ch(Q_j(E),\nabla^{Q_j(E)}_1,\tau)\notag
\\
&=-d\int_{0}^{1}\frac{1}{8\pi^2}{\det}^{1/2}\Big(\frac{\theta_j(R_u/(4\pi^2),\tau)}{\theta_j(0,\tau)}\Big)\tr\Big[g^{-1}dg
\frac{\theta_j'(R_u/(4\pi^2),\tau)}{\theta_j(R_u/(4\pi^2),\tau)}\Big]du.\notag
\end{align}

However, since $\tr\big[(g^{-1}dg)^n\big]$ vanishes for any positive even integer $n$ (see \cite[(1.40)]{MR1864735}), we
deduce from \eqref{4} that
for $j=1,2,3$,
\begin{equation}
{\det}^{1/2}\Big(\frac{\theta_j(R_u/(4\pi^2),\tau)}{\theta_j(0,\tau)}\Big)
=\exp\Big({\frac{1}{2}\tr\log}\frac{\theta_j(R_u/(4\pi^2),\tau)}{\theta_j(0,\tau)}\Big)=1.
\end{equation}

Therefore we have for $j=1,2,3$,
\begin{equation*}
\begin{split}
&\ch\big(Q_j(E)_v,\nabla^{Q_j(E)_v}_0,\tau\big)-\ch\big(Q_j(E)_v,\nabla^{Q_j(E)_v}_1,\tau\big)
\\
&\hspace{14em}=d\ch\big(Q_j(E)_v,g^{Q_j(E)_v},d,\tau\big),
\end{split}
\end{equation*}
where
\begin{equation}\label{7}
\begin{split}
&\ch(Q_1(E)_v,g^{Q_1(E)_v},d,\tau)
\\
&\quad=-\frac{2^{N/2}}{8\pi^2}\int_{0}^{1}\tr\Big[g^{-1}dg
\frac{\theta_1'(R_u/(4\pi^2),\tau)}{\theta_1(R_u/(4\pi^2),\tau)}\Big]du.
\end{split}
\end{equation}
and
\begin{equation}\label{cs}
\begin{split}
&\ch(Q_j(E)_v,g^{Q_j(E)_v},d,\tau)
\\
&\quad=-\frac{1}{8\pi^2}\int_{0}^{1}\tr\Big[g^{-1}dg
\frac{\theta_j'(R_u/(4\pi^2),\tau)}{\theta_j(R_u/(4\pi^2),\tau)}\Big]du, \quad j=2,3.
\end{split}
\end{equation}
As explained in \cite[Section 1.8]{MR1864735},
\begin{equation*}
\ch(Q_j(E)_v,g^{Q_j(E)_v},d,\tau), \ j=1,2,3
\end{equation*}are closed, and each determines an element in $H^{4i-1}(M,\mathbb{C})[[q^{1/2}]]$  depending only on
the homotopy class $[g]$.

Let
\begin{align*}
&\Gamma_0(2)=
\Big\{
\begin{pmatrix}
a & b
\\
c & d
\end{pmatrix}
\,\Big|\, a,b,c,d\in\mathbb{Z}, ad-bc=1
\Big\}
\end{align*}
as usual be the modular group, and
\begin{align}
&\Gamma_0(2)=
\Big\{
\begin{pmatrix}
a & b
\\
c & d
\end{pmatrix}
\in \SL_2(\mathbb{Z})
\,\Big|\, c\equiv 0\ ({\rm mod}\ 2)
\Big\},
\\
&\Gamma^0(2)=
\Big\{
\begin{pmatrix}
a & b
\\
c & d
\end{pmatrix}
\in \SL_2(\mathbb{Z})
\,\Big|\, b\equiv 0\ ({\rm mod}\ 2)
\Big\},
\\
\Gamma_\theta&=
\Big\{
\begin{pmatrix}
a & b
\\
c & d
\end{pmatrix}
\in \SL_2(\mathbb{Z})
\,\Big|\,
\begin{pmatrix}
a & b
\\
c & d
\end{pmatrix}\equiv
\begin{pmatrix}
1 & 0
\\
0 & 1
\end{pmatrix}
\text{or}
\begin{pmatrix}
0 & 1
\\
1 & 0
\end{pmatrix}
({\rm mod}\ 2)
\Big\}.
\end{align}
be the three modular subgroups of $\SL_2(\mathbb{Z})$.

\begin{defn}
Let $\Gamma$ be a subgroup of $\SL_2(\mathbb{Z})$. A modular form over $\Gamma$ is
a holomorphic function $f(\tau)$ on $\mathbb{H}\cup \{\infty\}$ such that for any
\begin{equation*}
\mathscr{G}=
\begin{pmatrix}
a & b
\\
c & d
\end{pmatrix}
\in \Gamma,
\end{equation*}
the following property holds,
\begin{equation*}
f(\mathscr{G}\tau):=f\Big(\frac{a\tau+b}{c\tau+d}\Big)=\chi(\mathscr{G})(c\tau+d)^k f(\tau),
\end{equation*}
where $\chi:\Gamma\longrightarrow \mathbb{C}^*$ is a character of $\Gamma$ and $k$ is called the weight of $f$.
\end{defn}

If $\omega$ is a differential form on $M$, we denote by $\omega^{(i)}$ the degree $i$ component of $\omega$.

\begin{prop}[{Compare with \cite[Theorem 5.1]{MR2495834}}]\label{modular}
For any integer $i\geq 2$,
\begin{equation*}
\{\ch(Q_j(E)_v,g^{Q_j(E)_v},d,\tau)\}^{(4i-1)},\quad j=1,2,3
\end{equation*}
are modular forms of weight $2i$ over $\Gamma_0(2)$, $\Gamma^0(2)$ and $\Gamma_\theta$, respectively.
\end{prop}
\begin{proof}
Using the transformation formulas \cite[(3.29)-(3.31),(4.6)]{MR2495834}, we deduce directly from \eqref{7} and \eqref{cs} that
\begin{equation}
\begin{split}\label{1}
\ch(Q_1(E)_v,g^{Q_1(E)_v},d,\tau+1)&=\ch(Q_1(E)_v,g^{Q_1(E)_v},d,\tau),
\\
\ch(Q_2(E)_v,g^{Q_2(E)_v},d,\tau+1)&=\ch(Q_3(E)_v,g^{Q_3(E)_v},d,\tau).
\end{split}
\end{equation}
and that for any integer $i\geq 1$,
\begin{equation}
\begin{split}\label{64}
&\Big\{\ch\Big(Q_1(E)_v,g^{Q_1(E)_v},d,-\frac{1}{\tau}\Big)\Big\}^{(4i-1)}
\\
&=2^{N/2}\Big\{\tau^{2i}\ch(Q_2(E)_v,g^{Q_2(E)_v},d,\tau)
-\frac{\tau\sqrt{-1}}{24\pi}c_3(E_\mathbb{C},g,d)\Big\}^{(4i-1)}\ ,
\end{split}
\end{equation}
and
\begin{equation}
\begin{split}\label{2}
&\Big\{\ch\Big(Q_3(E)_v,g^{Q_3(E)_v},d,-\frac{1}{\tau}\Big)\Big\}^{(4i-1)}
\\
&=\Big\{\tau^{2i}\ch(Q_3(E)_v,g^{Q_3(E)_v},d,\tau)-\frac{\tau\sqrt{-1}}{24\pi}c_3(E_\mathbb{C},g,d)\Big\}^{(4i-1)}\ .
\end{split}
\end{equation}

Recall that the generators of $\Gamma_0(2)$ are $T$, $ST^2ST$,
the generators of $\Gamma^0(2)$ are $STS$, $T^2STS$ and the generators of $\Gamma_\theta$ are $S$, $T^2$ (see \cite{MR808396}),
where
\begin{equation}
S=
\begin{pmatrix}
0 & -1
\\
1 & 0
\end{pmatrix},\quad
T=
\begin{pmatrix}
1 & 1
\\
0 & 1
\end{pmatrix}.
\end{equation}
Due to the above fact, the proposition now follows from \eqref{1}-\eqref{2} in a standard way (see, e.g., \cite[(4.13)]{MR2495834}).
\end{proof}

\subsection{Elliptic genera in odd dimensions}

Set (see \cite{MR970288,MR1372798})
\begin{align}
\Theta_{1}(TM)_v&=\bigotimes_{n=1}^\infty S_{q^n}(\widetilde{T_\mathbb{C}M})\otimes
\bigotimes_{n=1}^{\infty}\Lambda_{q^n}(\widetilde{T_\mathbb{C}M}),
\\
\Theta_{2}(TM)_v&=\bigotimes_{n=1}^\infty S_{q^n}(\widetilde{T_\mathbb{C}M})\otimes
\bigotimes_{n=1}^{\infty}\Lambda_{-q^{n-{1/2}}}(\widetilde{T_\mathbb{C}M}),
\\
\Theta_{3}(TM)_v&=\bigotimes_{n=1}^\infty S_{q^n}(\widetilde{T_\mathbb{C}M})\otimes
\bigotimes_{n=1}^{\infty}\Lambda_{q^{n-{1/2}}}(\widetilde{T_\mathbb{C}M}),
\\
\Theta(TM)_v&=\bigotimes_{n=1}^\infty S_{q^n}(\widetilde{T_\mathbb{C}M}).
\end{align}
We introduce the odd analogues of the Landweber-Stong forms and the Witten forms as follows.

\begin{defn}[{Compare with \cite[Definition 3.1]{MR2495834}}]
We call
\begin{equation}
\begin{split}
\Phi_{{\rm L}}(\nabla^{TM},g,d,\tau)&=2^{\,(\dim M-1)/2}\,\widehat{L}(TM,\nabla^{TM})
\\
&\hspace{-20pt}\cdot\ch(\Theta_1(TM)_v,\nabla^{\Theta_1(TM)_v},\tau)\cdot\ch(Q_1(E)_v,g^{Q_1(E)_v},d,\tau)
\end{split}
\end{equation}
the Landweber-Stong type form of $M$ associated to $\nabla^{TM}$, $d$ and $g$. Also,
we call
\begin{equation}
\begin{split}
\Phi_{{\rm W}}(\nabla^{TM},g,d,\tau)
&=\widehat{A}(TM,\nabla^{TM})\ch(\Theta_2(TM)_v,\nabla^{\Theta_2(TM)_v},\tau)
\\
&\hspace{40pt}\cdot\ch(Q_2(E)_v,g^{Q_2(E)_v},d,\tau)
\end{split}
\end{equation}
\begin{equation}
\begin{split}
\Phi'_{{\rm W}}(\nabla^{TM},g,d,\tau)
&=\widehat{A}(TM,\nabla^{TM})\ch(\Theta_3(TM)_v,\nabla^{\Theta_3(TM)_v},\tau)
\\
&\hspace{40pt}\cdot\ch(Q_3(E)_v,g^{Q_3(E)_v},d,\tau),
\end{split}
\end{equation}
\begin{equation}
\begin{split}
\Psi_{{\rm W},\,j}(\nabla^{TM},g,d,\tau)
&=\widehat{A}(TM,\nabla^{TM})\ch(\Theta(TM)_v,\nabla^{\Theta(TM)_v},\tau)
\\
&\hspace{30pt}\cdot\ch(Q_j(E)_v,g^{Q_j(E)_v},d,\tau),\hspace{15pt} j=1,2,3
\end{split}
\end{equation}
the Witten type forms of $M$ associated to $\nabla^{TM}$, $d$ and $g$.
\end{defn}

Applying the Chern-Weil theory, we can express the Landweber-Stong type forms and the Witten type forms
in terms of theta functions and curvatures as in \cite[Proposition 3.1]{MR2495834} (see also \cite{MR1396769,MR1372798,MR1331972}).

Combining Proposition \ref{modular} and \cite[Proposition 3.2]{MR2495834}, we
establish the following modularities.

\begin{prop}[{Compare with \cite[Proposition 3.2]{MR2495834}}]
\
\begin{enumerate}[{\rm (i)}]
\item If $c_3(E_\mathbb{C},g,d)=0$, then for any integer $i\geq 1$,
\begin{equation*}
\{\Phi_{{\rm L}}(\nabla^{TM},g,d,\tau)\}^{(4i-1)},\quad \{\Phi_{{\rm W}}(\nabla^{TM},g,d,\tau)\}^{(4i-1)}
\end{equation*}
and $\{\Phi'_{{\rm W}}(\nabla^{TM},g,d,\tau)\}^{(4i-1)}$ are modular forms of
weight $2i$ over $\Gamma_0(2)$, $\Gamma^0(2)$ and $\Gamma_\theta$ respectively.

\item If $c_3(E_\mathbb{C},g,d)=0$ and the first Pontryagin form $p_1(TM,\nabla^{TM})=0$, then for any integer $i\geq 1$,
\begin{equation*}
\{\Psi_{{\rm W},\,j}(\nabla^{TM},g,d,\tau)\}^{(4i-1)},\quad j=1,2,3
\end{equation*}
are modular forms of weight $2i$ over $\Gamma_0(2)$, $\Gamma^0(2)$ and $\Gamma_\theta$ respectively.
\end{enumerate}
\end{prop}

We now assume that $M$ is a $(4k-1)$-dimensional closed oriented smooth manifold.
Let $[M]$ be the fundamental class of $M$. Set
\begin{align*}
\phi_{{\rm L}}(M,[g],\tau)&=\langle\Phi_{{\rm L}}(\nabla^{TM},g,d,\tau),[M]\rangle,
\\
\phi_{{\rm W}}(M,[g],\tau)&=\langle\Phi_{{\rm W}}(\nabla^{TM},g,d,\tau),[M]\rangle,
\\
\phi'_{{\rm W}}(M,[g],\tau)&=\langle\Phi'_{{\rm W}}(\nabla^{TM},g,d,\tau),[M]\rangle,
\\
\psi_{{\rm W},\,j}(M,[g],\tau)&=\langle\Psi_{{\rm W},\,j}(\nabla^{TM},g,d,\tau),[M]\rangle, \quad j=1,2,3.
\end{align*}

\begin{defn}
We call $\phi_{{\rm L}}(M,[g],\tau)$, $\phi_{{\rm W}}(M,[g],\tau)$ and $\phi'_{{\rm W}}(M,[g],\tau)$
the {\it elliptic genera of the pair} $(M, [g])$, and
call $\phi_{{\rm W}}(M,[g],\tau)$, $\phi'_{{\rm W}}(M,[g],\tau)$, $\psi_{{\rm W},\,j}(M,[g],\tau)$, $j=1,2,3$
the {\it Witten genera of the pair} $(M, [g])$.
\end{defn}

\begin{thm} Assume $c_3(M,[g])=0$. We have \newline
{\rm (i)} the elliptic genera of the pair $(M, [g])$

\begin{equation*}
\phi_{{\rm L}}(M,[g],\tau),\ \phi_{{\rm W}}(M,[g],\tau)\ \text{and}\ \phi'_{{\rm W}}(M,[g],\tau)
\end{equation*}
are modular forms of weight $(\dim M+1)/2$ over $\Gamma_0(2)$, $\Gamma^0(2)$, and $\Gamma_\theta$, respectively;

\vspace{5pt}

\noindent{\rm (ii)} if the first Pontryagin class $p_1(M)=0$, then the Witten genera of the pair $(M, [g])$
\begin{equation*}
\psi_{{\rm W},\,j}(M,[g],\tau),\quad j=1,2,3
\end{equation*}
are modular forms of weight $(\dim M+1)/2$ over $\Gamma_0(2)$, $\Gamma^0(2)$, and $\Gamma_\theta$, respectively.

\end{thm}

Assume that $M$ is spin. Then by the index formula \eqref{35}, we can interpret the above elliptic genera and Witten genera of the pair $(M, [g])$ on $(4k-1)$-dimensional manifolds
analytically as the indices of the twisted Toeplitz operators as follows,
\begin{align*}
\phi_{{\rm L}}(M,[g],\tau)&=-\ind (\mathcal{T}\otimes\Delta(TM)\otimes\Theta_1(TM)_v\otimes (Q_1(E)_v,g^{Q_1(E)_v})),
\\
\phi_{{\rm W}}(M,[g],\tau)&=-\ind (\mathcal{T}\otimes\Theta_2(TM)_v\otimes (Q_2(E)_v,g^{Q_2(E)_v})),
\\
\phi'_{{\rm W}}(M,[g],\tau)&=-\ind (\mathcal{T}\otimes\Theta_3(TM)_v\otimes (Q_3(E)_v,g^{Q_3(E)_v})),
\\
\psi_{{\rm W},\,j}(M,[g],\tau)&=-\ind (\mathcal{T}\otimes\Theta(TM)_v\otimes (Q_j(E)_v,g^{Q_j(E)_v})),\quad j=1,2,3.
\end{align*}

\section{Witten rigidity on odd dimensional manifolds}\label{sec3}

\subsection{An $S^1$-equivariant index theorem for Toeplitz operators}

Let $M$ be an odd dimensional closed smooth spin Riemannian manifold which
admits a circle action. Without loss of generality, we may assume that $S^1$
acts on $M$ isometrically and preserves the spin structure of $M$.

Let $\mathcal{E}$ be an $S^1$-equivariant complex vector bundle over $M$ carrying an $S^1$-invariant Hermitian
connection. Then the associated twisted Dirac operator $D\otimes \mathcal{E}$ is $S^1$-equivariant, which
implies that the corresponding orthogonal projection $P_+$  is also $S^1$-equivariant.

In addition, we assume $g:M\longrightarrow {\rm GL}(N,\mathbb{C})$ is $S^1$-invariant, i.e.,
\begin{equation}\label{11}
g(hx)=g(x),\quad \text{for any}\ h\in S^1\ \text{and}\ x\in M.
\end{equation}
Thus the twisted Toeplitz operator $\mathcal{T}\otimes \mathcal{E}\otimes (\mathbb{C}^N|_M,g)$ is $S^1$-equivariant.

Let $M^{S^1}$ denote the fixed submanifold of the circle action on $M$. In general, $M^{S^1}$ is not
connected. We fix a connected component $M^{S^1}_\alpha$ of $M^{S^1}$, and omit
the subscript $\alpha$ if there is no confusion.

Let $N$ denote the normal bundle to $M^{S^1}$ in $M$, which can be identified as the orthogonal
complement of $TM^{S^1}$ in $TM\big|_{M^{S^1}}$.
Then we have the following $S^1$-equivariant decomposition when restricted upon $M^{S^1}$,
\begin{equation}
TM\big|_{M^{S^1}}=N_{m_1}\oplus \cdots \oplus N_{m_l} \oplus TM^{S^1},
\end{equation}
where each $N_\gamma$, $\gamma=m_1,\cdots, m_l$, is a complex vector bundle such that $h\in S^1$ acts on it by $h^\gamma$.
To simplify the notation, we will write that
\begin{equation}\label{25}
TM\big|_{M^{S^1}}=\bigoplus_{\gamma\neq0} N_\gamma \oplus TM^{S^1}\ ,
\end{equation}
where $N_\gamma$ is a complex vector bundle such that $h\in S^1$ acts on it by $h^\gamma$ with
$\gamma\in \mathbb{Z}\backslash\{0\}$. Clearly, $N=\bigoplus_{\gamma\neq0} N_\gamma$.
From now on, we will regard $N$ as a complex vector bundle.
Let
\begin{equation*}
2\pi\sqrt{-1}x_\gamma^{\,j},\ j=1,\cdots,\dim N_\gamma
\end{equation*}
be the Chern roots of $N_\gamma$. Let
\begin{equation*}
\pm 2\pi\sqrt{-1}y_j,\ j=1,\cdots, (\dim M^{S^1}-1)/2
\end{equation*}
be the Chern roots of $TM^{S^1}\otimes\mathbb{C}$.

Similarly, let
\begin{equation}
\mathcal{E}\big|_{M^{S^1}}=\bigoplus_{\nu} \mathcal{E}_\nu
\end{equation}
be the $S^1$-equivariant decomposition of the restrictions of $\mathcal{E}$ over $M^{S^1}$, where
$\mathcal{E}_\nu$ is a complex vector bundle such that $h\in S^1$ acts on it by $h^\nu$ with
$\nu\in \mathbb{Z}$. We denote by
\begin{equation*}
2\pi\sqrt{-1}w_\nu^{\,j},\ j=1,\cdots,\dim \mathcal{E}_\nu
\end{equation*}
the Chern roots of $\mathcal{E}_\nu$.

For $f(\cdot)$ a holomorphic function, we denote by $f(y)(TM^{S^1})=\prod_j f(y_j)$ the
symmetric polynomial that gives characteristic class of $TM^{S^1}$, and we use the same
notation for $N_\gamma$.

The following equivariant index formula is an immediate consequence of the
odd equivariant index theorem for Toeplitz operators of Fang  \cite[Theorem 4.3]{MR2200374} and Liu-Wang \cite[Theorem 2.3]{MR2532715}.

\begin{prop}[{\cite[(2.5)]{MR2532715}}]\label{fixed}
Let $h=e^{2\pi\sqrt{-1}t}$,  $t\in [0,1]\backslash\mathbb{Q}$ be a topological generator of $S^1$. Then
\begin{equation}
\begin{split}\label{34}
&\ind\big(h,\mathcal{T}\otimes\mathcal{E}\otimes (\mathbb{C}^{N}|_{M},g)\big)
=-\,\Big\langle
\ch(\mathbb{C}^{N}|_{M},g,d)\,\frac{\pi y}{\sin (\pi y)}(TM^{S^1})
\\
&\hspace{20pt}\cdot\prod_{\gamma}\frac{1}{2\sqrt{-1}\sin\pi(x_\gamma+\gamma\,t)}(N_\gamma)
\cdot \sum_\nu\sum_j e^{2\pi\sqrt{-1}(w^j_\nu+\nu\,t)},\big[M^{S^1}\big]\Big\rangle\ ,
\end{split}
\end{equation}
where $\big[M^{S^1}\big]$ is the fundamental class of $M^{S^1}$ which carries the orientation compatible with
that of $M$ and $N$.
\end{prop}

As a direct application of Proposition \ref{fixed}, we deduce an odd analogue of the rigidity of the signature operator.
\begin{thm}\label{36}
$\mathcal{T}\otimes \Delta (TM)\otimes (\mathbb{C}^N|_M,g)$ is rigid.
\end{thm}
\begin{proof}
For $z\in \mathbb{C}$, set
\begin{equation}
\begin{split}
f(z)&=-2^{\,(\dim M^{S^1}-1)/2}\Big\langle
\ch(\mathbb{C}^{N}|_{M},g,d)\,\frac{\pi y}{\tan (\pi y)}(TM^{S^1})
\\
&\hspace{70pt}
\cdot\prod_{\gamma}
\frac{z^\gamma e^{\pi\sqrt{-1}x_\gamma}+ e^{-\pi\sqrt{-1}x_\gamma}}
{z^\gamma e^{\pi\sqrt{-1}x_\gamma}-e^{-\pi\sqrt{-1}x_\gamma}}(N_\gamma)
,\big[M^{S^1}\big]\Big\rangle\ .
\end{split}
\end{equation}
Then $f(z)$ is a rational function and has no poles on $\mathbb{C}\setminus S^1$.

But by \eqref{25} and \eqref{34}, we see that $f$ coincides with the continuous function
$\ind\big(z,\mathcal{T}\otimes\mathcal{E}\otimes (\mathbb{C}^{N}|_{M},g)\big)$ on the dense subset which
consists of the topological generators of $S^1$. Thus $f$ must be bounded on $S^1$.

Now $f(z)$ is constant on $\mathbb{C}$ due to the fact that $\lim_{z\rightarrow \infty} f(z)$ exists.
\end{proof}

\subsection{Witten rigidity in odd dimensions}
Let $V$ be an $S^1$-equivariant real spin vector bundle over $M$.
Let $\Delta(V)$ be the corresponding spinor bundle.

Let $g:M\longrightarrow \SO(N)$ be an $S^1$-invariant smooth map from $M$ to $\SO(N)$ with $N$
a positive even integer large enough.
Let $E$ denote the trivial real vector bundle of rank $N$ over $M$,
which is equipped with the canonical trivial metric and trivial connection $d$.
Set (cf. \eqref{50})
\begin{equation}
\begin{split}
Q_1(E)&=\Delta(E)\otimes\bigotimes_{n=1}^{\infty}\Lambda_{q^n}(E_\mathbb{C})\ ,
\\
Q_2(E)&=\bigotimes_{n=1}^{\infty}\Lambda_{-q^{n-1/2}}(E_\mathbb{C})\ ,
\quad\
Q_3(E)=\bigotimes_{n=1}^{\infty}\Lambda_{q^{n-1/2}}(E_\mathbb{C})\ .
\end{split}
\end{equation}
Let $g^{Q_j(E)}$, $j=1,2,3$ be the actions on $Q_j(E)$ respectively induced from the action of $g$ on $E$.

Following \cite{MR1396769}, we introduce the following elements in $K(M)[[q^{1/2}]]$,
\begin{align}
\Theta_{1}(TM|V)&=\bigotimes_{n=1}^\infty S_{q^n}(T_\mathbb{C}M)\otimes
\bigotimes_{n=1}^{\infty}\Lambda_{q^n}(V_\mathbb{C}),\label{38}
\\
\Theta_{2}(TM|V)&=\bigotimes_{n=1}^\infty S_{q^n}(T_\mathbb{C}M)\otimes
\bigotimes_{n=1}^{\infty}\Lambda_{-q^{n-{1/2}}}(V_\mathbb{C}),
\\
\Theta_{3}(TM|V)&=\bigotimes_{n=1}^\infty S_{q^n}(T_\mathbb{C}M)\otimes
\bigotimes_{n=1}^{\infty}\Lambda_{q^{n-{1/2}}}(V_\mathbb{C}).\label{39}
\end{align}

For simplicity, denote
\begin{equation*}
\Theta_{1}(TM|TM),\ \Theta_{2}(TM|TM)\ \text{and}\ \Theta_{3}(TM|TM)
\end{equation*}
 by
\begin{equation*}
\Theta_{1}(TM),\ \Theta_{2}(TM)\ \text{and}\ \Theta_{3}(TM),
\end{equation*}
respectively.

If $V$ is even dimensional, let $\Delta(V)=\Delta_+(V)\oplus \Delta_-(V)$ be
the natural $\mathbb{Z}_2$-grading. Set
\begin{equation}\label{40}
\Theta(TM|V)=\bigotimes_{n=1}^\infty S_{q^n}(T_\mathbb{C}M)\otimes
\bigotimes_{n=1}^{\infty}\Lambda_{-q^n}(V_\mathbb{C}).
\end{equation}

Let $H^*_{S^1}(M,\mathbb{Z})=H^*(M\times_{S^1}ES^1,\mathbb{Z})$ denote the $S^1$-equivariant cohomology group
of $M$, where $ES^1$ is the universal $S^1$-principal bundle over the classifying space $BS^1$ of $S^1$. So
$H^*_{S^1}(M,\mathbb{Z})$ is a module over $H^*(BS^1,\mathbb{Z})$ induced by the projection
$\pi:M\times_{S^1}ES^1\rightarrow BS^1$. Recall that
\begin{equation}
H^*(BS^1,\mathbb{Z})=\mathbb{Z}[[u]]
\end{equation}
with $u$ being a generator of degree $2$.

The $S^1$-equivariant characteristic class of an $S^1$-bundle $W$ over $M$ by definition is the usual
characteristic class of the bundle $W\times_{S^1}ES^1$ over $M\times_{S^1}ES^1$.
Let $p_1(\cdot)_{S^1}$ denote the first $S^1$-equivariant pontrjagin class.

We suppose that there exists some integer $n\in \mathbb{Z}$ such that
\begin{equation}\label{22}
p_1(V)_{S^1}-p_1(TM)_{S^1}=n\cdot \pi^*u^2.
\end{equation}
Following \cite{MR1331972}, we call $n$ the anomaly to rigidity.

The following theorems generalize
the Witten rigidity theorems and vanishing theorems \cite[Corollary 3.1]{MR1331972} to the case of odd dimensional manifolds.
They are similar to \cite[Theorem 2.5 and Corollary 3.6]{MR2532715}, while
without putting restriction on the dimension of fixed point set. Instead we put some topological conditions on $g$.

\begin{thm}\label{24}
Suppose $g_*=1$, $c_3(M, [g])=0$ and \eqref{22} holds. Then we have
\begin{enumerate}[{\rm (i)}]
\item If $n=0$, then following $S^1$-equivariant Toeplitz operators are rigid,
\begin{equation*}
\begin{split}
&\mathcal{T}\otimes \Delta(V)\otimes \Theta_1(TM|V)\otimes (Q_1(E),g^{Q_1(E)}),
\\
&\mathcal{T}\otimes \Theta_2(TM|V)\otimes (Q_2(E),g^{Q_2(E)}),
\\
&\mathcal{T}\otimes \Theta_3(TM|V)\otimes (Q_3(E),g^{Q_3(E)}),
\\
&\mathcal{T}\otimes (\Delta_+(V)-\Delta_-(V))\otimes \Theta(TM|V)\otimes (Q_j(E),g^{Q_j(E)}),\ j=1,2,3.
\end{split}
\end{equation*}

\item If $n<0$, then the equivariant indices of the above operators all vanish.
\end{enumerate}

\end{thm}

\begin{proof}
Combining \cite[Theorem 1.2]{MR781735} and Theorem \ref{30}, which will be given in the next subsection,
we obtain the virtual version of the theorem. Since the equivariant index of the virtual case and
that of the non-virtual one differ by a constant (depending on $q$), we complete the proof of Theorem \ref{24}.
\end{proof}

\begin{thm}\label{37}
Suppose $g_*=1$, $c_3(M, [g])=0$ and \eqref{22} holds. Then we have
\begin{enumerate}[{\rm (i)}]
\item If $n=0$, then following $S^1$-equivariant Toeplitz operators are rigid,
\begin{equation*}
\mathcal{T}\otimes \Theta_2(TM|V)\otimes (E_\mathbb{C},g),\quad
\mathcal{T}\otimes \Theta_3(TM|V)\otimes (E_\mathbb{C},g).
\end{equation*}

\item If $n<0$, then the equivariant indices of the above operators vanish.
\end{enumerate}

\end{thm}

\begin{proof}
Suppose $\Theta_2(TM|V)$ and $Q_2(E)$ admit formal Fourier
expansion in $q^{1/2}$ as
\begin{equation}
\Theta_2(TM|V)=\sum_{j=0}^\infty A_j\,q^{j/2}\,,
\quad Q_2(E)=\sum_{j=0}^\infty B_j\,q^{j/2},
\end{equation}
with $A_j$'s, $B_j$'s being elements in $K(M)$.
In particular, we verify that
\begin{equation}
\begin{split}\label{33}
A_0&=\mathbb{C}|_{M},
\quad A_1=-V_\mathbb{C},
\quad A_2=T_\mathbb{C}M\oplus \Lambda^2 (V_\mathbb{C}),
\\
B_0&=\mathbb{C}|_{M},
\quad B_1=-E_\mathbb{C},
\quad B_2=\Lambda^2 (E_\mathbb{C}).
\end{split}
\end{equation}

We will use the following convention for the sake of simplicity.

\vspace{2pt}

\noindent{\bf Convention} We will say an $S^1$-equivariant twisted Toeplitz operator $P$ has
\emph {good property}, if $P$ is rigid when $n=0$, and if $P$ has vanishing property when $n<0$.

\vspace{2pt}

Applying Theorem \ref{24} and picking up the $q^{\frac{j+1}{2}}$-coefficient in the expansion of the operator
\begin{equation*}
\mathcal{T}\otimes \Theta_2(TM|V)\otimes (Q_2(E),g^{Q_2(E)}),
\end{equation*}
we see that the equivariant Toeplitz operator
\begin{equation}\label{45}
\mathcal{T}\otimes A_0\otimes (B_{j+1}, g^{B_{j+1}}) \cdots +\mathcal{T}\otimes A_j\otimes (B_1, g^{B_1})
\end{equation}
has the good property.

First note that the operator $\mathcal{T}\otimes A_0\otimes (B_{j}, g^{B_{j}})$ has the good property for all $j\geq 1$ due to the theorem of Liu-Wang \cite[Theorem 2.4]{MR2532715}. Here no conditions about Chern classes of $(B_{j}, g^{B_{j}})$ are needed.

From Theorem \ref{24}, we see that under the condition
\begin{equation*}
g_*=1,\ c_3(M, [g])=0,
\end{equation*}
taking $j=1$, the operator
\begin{equation*}
\mathcal{T}\otimes A_0\otimes (B_{2}, g^{B_{2}})+\mathcal{T}\otimes A_1\otimes (B_1, g^{B_1})
\end{equation*}
has the good property. Therefore, under the condition $g_*=1$, $c_3(M, [g])=0$, the equivariant Toeplitz operator
$$\mathcal{T}\otimes A_1\otimes (B_1, g^{B_1})$$
has the good property.
We will show that the operator $$\mathcal{T}\otimes A_1\otimes (B_{j}, g^{B_{j}})$$ has the good property for all $j\geq 1$.
It suffices to show that
\begin{equation}\label{46}
(g^{B_j})_*=1, \ c_3\Big(M,\big[g^{B_j}\big]\Big)=0,\quad \text{for}\ j\geq 2.
\end{equation}

Since $c_3(M, [g])=0$, taking the degree $3$ component of \eqref{cs}, we get
\begin{equation}
\sum_{j=1}^\infty c_3\Big(M,\big[g^{B_j}\big]\Big)\,q^{j/2}=0,
\end{equation}
which implies $c_3\big(M,\big[g^{B_j}\big]\big)=0$ for all $j\geq 2$.

Observe that $B_j$ is the sum of bundles of the form
\begin{equation*}
\Lambda^{n_1}E_\CC\otimes \Lambda^{n_2} E_\CC\otimes\cdots\otimes\Lambda^{n_k} E_\CC,\quad 1\leq n_1,\cdots,n_k\leq \dim E.
\end{equation*}
We have the following lemma.

\begin{lem}\label{61}
Let $V_1, \cdots, V_n, V$ be complex vector spaces. Let
\begin{equation*}
h_i: M\longrightarrow {\rm Aut}\,(V_i),\ i=1, \cdots, n \quad \text{and}\quad h: M\longrightarrow {\rm Aut}\,(V)
\end{equation*}
be cocyles in $K^{-1}(M)$. If the induced maps on the fundamental groups satisfy
$(h_i)_*=1$, $i=1,\cdots,n$ and $h_*=1$, then we have
\vspace{-5pt}
\begin{equation*}
(h_1\otimes \cdots \otimes h_n)_*=1 \quad  \text{and}\quad \Big(\overbrace{h\wedge \cdots \wedge h}^m\Big)_*=1,
\end{equation*}
where $h_1\otimes \cdots \otimes h_n\in {\rm Aut}\,(V_1\otimes\cdots \otimes V_n)$ is defined by
\begin{equation*}
(h_1\otimes \cdots \otimes h_n)(v_1\otimes\cdots \otimes v_n)=h_1(v_1)\otimes\cdots \otimes h_n(v_n),\quad v_i\in V_i,
\end{equation*}
and $\overbrace{h\wedge \cdots \wedge h}^m\in {\rm Aut}\,(\Lambda^m(V))$ is defined by
\vspace{-5pt}
\begin{equation*}
\Big(\overbrace{h\wedge \cdots \wedge h}^m\Big)(u_1\wedge\cdots \wedge u_m)=h(u_1)\wedge\cdots \wedge h(u_m),\quad u_i\in V.
\end{equation*}
\end{lem}
\begin{proof}
For a loop $\gamma$ representing an element in $\pi_1(M)$, since $(h_i)_*=1$, $i=1, \cdots n$, and $h_*=1$, there
exist homotopies $\eta_i(t)$, $\eta(t)$, $t\in [0,1]$ such that
\begin{equation*}
\eta_i(0)=h_i\circ \gamma, \ \eta_i(1)={\rm const}; \ \eta(0)=h\circ \gamma,\ \eta(1)={\rm const}.
\end{equation*}
Therefore, $\eta_1(t)\otimes\cdots\otimes\eta_n(t)$ connects the loop
\begin{equation*}
(h_1\otimes  \cdots \otimes h_n)\circ \gamma=(h_1\circ \gamma) \otimes \cdots \otimes (h_n\circ \gamma)
\end{equation*}
to a constant loop in ${\rm Aut}\,(V_1\otimes\cdots \otimes V_n)$. Similarly $\overbrace{\eta(t)\wedge\cdots\wedge\eta(t)}^m$ connects the loop
\vspace{-10pt}
\begin{equation*}
\Big(\overbrace{h\wedge \cdots \wedge h}^m\Big)\circ \gamma=\overbrace{(h\circ \gamma) \wedge \cdots \wedge (h_n\circ \gamma)}^m
\end{equation*}
to a constant loop in ${\rm Aut}\,(\Lambda^m(V))$.
\end{proof}

As we have assumed $g_*=1$,  the above lemma tells us that
$(g^{B_j})_*=1$. Therefore \eqref{46} holds and the operator $$\mathcal{T}\otimes A_1\otimes (B_{j}, g^{B_{j}})$$ has the good property for all $j\geq 1$.

Now taking $j=2$ in \eqref{45}, we see that the operator
\begin{equation*}
\mathcal{T}\otimes A_0\otimes (B_{3}, g^{B_{3}})+\mathcal{T}\otimes A_1\otimes (B_2, g^{B_2})+\mathcal{T}\otimes A_2\otimes (B_1, g^{B_1}).
\end{equation*} has the good property. Since from the above discussion the operators
\begin{equation*}
\mathcal{T}\otimes A_0\otimes (B_{3}, g^{B_{3}}) \text{ and }\ \mathcal{T}\otimes A_1\otimes (B_2, g^{B_2})
\end{equation*}
both have the good property, we obtain that the operator
$$\mathcal{T}\otimes A_2\otimes (B_1, g^{B_1})$$ has the good property.
Therefore, due to \eqref{46}, the operator $$\mathcal{T}\otimes A_2\otimes (B_{j}, g^{B_{j}})$$ has the good property for all $j\geq 1$.

A standard induction procedure shows that $\mathcal{T}\otimes A_i\otimes B_j$ has the good property for any $i$, $j\geq 1$. In particular, we see $\mathcal{T}\otimes A_i\otimes B_1$ has the good property for all $i\geq 0$.
We finish the proof for $\mathcal{T}\otimes \Theta_2(TM|V)\otimes (E_\mathbb{C},g)$.

It is easy to see that with little modification, the above deduction still
applies to  the operator $\mathcal{T}\otimes \Theta_3(TM|V)\otimes (E_\mathbb{C},g)$.
\end{proof}

We would like to point out that we deduce $c_3\big(M,\big[g^{B_j}\big]\big)=0$ in \eqref{46} from (\ref{cs}) by taking the advantage of
the special positive energy representation that $Q_2(E)$ is constructed on.
An alternative deduction is by using the following proposition and the observation above Lemma \ref{61}.

\begin{prop}\label{60}
Let $V_1,\cdots, V_n$, $V$ be complex vector spaces. Let
\begin{equation*}
h_i: M\longrightarrow {\rm Aut}\,(V_i),\ i=1,\cdots,n \quad \text{and}\quad h: M\longrightarrow {\rm Aut}\,(V)
\end{equation*}
be cocyles in $K^{-1}(M)$.

\noindent{\rm (i)} The following equality holds,
\begin{equation}\label{54}
\ch(V_1\otimes\cdots\otimes V_n,h_1\otimes \cdots \otimes h_n,d)
=\sum_{i=1}^n\frac{\dim V_1\cdots\dim V_n}{\dim V_i}\ch(V_i,h_i,d).
\end{equation}
\noindent{\rm (ii)} For any $j$, $m\geq 1$,
\vspace{-10pt}
\begin{equation*}
c_{2j-1}\Big(\Lambda^m(V), \overbrace{h\wedge \cdots \wedge h}^m,d\Big)
\end{equation*}
is a constant multiples of $c_{2j-1}(V,h,d)$.

\end{prop}

\begin{proof}
{\rm (i)} Let (cf. \eqref{53})
\begin{equation*}
\nabla^{V_i}_u= d+ u\,h_i^{-1}dh_i,\quad u\in [0,1]
\end{equation*}
be the curves of connections on $V_i$, $i=1,2$, respectively. Then $V_1\otimes V_2$ naturally carries a curve of
tensor connections
\begin{equation*}
\nabla^{V_1\otimes V_2}_u= d+ u(h_1^{-1}dh_1\otimes {\rm id}+{\rm id}\otimes h_2^{-1}dh_2),\quad u\in [0,1].
\end{equation*}
We verify directly that the curvature $R^{V_1\otimes V_2}_u$ of $\nabla^{V_1\otimes V_2}_u$ is given by
\begin{equation}\label{55}
R^{V_1\otimes V_2}_u=(u^2-u)\big((h_1^{-1}dh_1)^2\otimes {\rm id}+{\rm id}\otimes(h_2^{-1}dh_2)^2\big).
\end{equation}
Using the explicit formula of the Chern-Simons form \cite[(1.25)]{MR1231957}, we get
\begin{align}\label{56}
\ch(V_1\otimes V_2, h_1\otimes h_2,d)&=\frac{1}{2\pi\sqrt{-1}}\int_{0}^1\tr\Big[(h_1^{-1}dh_1\otimes {\rm id}+{\rm id}\otimes h_2^{-1}dh_2)
\notag\\
&\hspace{5em}\cdot\exp\big({\sqrt{-1}R^{V_1\otimes V_2}_u/(2\pi)}\big)\Big]du.
\end{align}
Since $\tr\big[(h^{-1}dh)^k\big]$ vanishes for any positive even integer $k$ (see \cite[(1.40)]{MR1864735}),
by \eqref{55} and \eqref{56}, we obtain \eqref{54} for the case $n=2$.
The proof of \eqref{54} for the general case follows in a similar way.

\

\noindent{\rm (ii)} We consider the curve of connections
\begin{equation*}
\nabla^{V}_u= d+ u\,h^{-1}dh, \quad u\in [0,1]
\end{equation*}
on $V$. Let $R^V_u=(u^2-u)(h^{-1}dh)^2$ be the curvature of $\nabla^{V}_u$.

$\nabla^V_u$ canonically
induces a connection $\nabla^{\Lambda_t(V)}$ on $\Lambda_t(V)$ for any $t\in \mathbb{C}$.
Furthermore, we can compute the Chern character form of
$(\Lambda_t(V),\nabla^{\Lambda_t(V)})$ as follows,
\begin{equation}\label{63}
\ch\big(\Lambda_t(V),\nabla^{\Lambda_t(V)}\big)=\det\Big(1+t\exp\Big(\frac{\sqrt{-1}}{2\pi}R_u^V\Big)\Big),\quad u\in [0,1].
\end{equation}

As in \cite[(2.8)]{MR2495834}, we deduce from \eqref{63} that for $t\in \mathbb{C}\backslash\{-1\}$,
\begin{equation}
\begin{split}\label{57}
&\frac{d}{du}\ch(\Lambda_t(V),\nabla^{\Lambda_t(V)})
\\
&=\frac{\sqrt{-1}}{2\pi}d\Big(\ch(\Lambda_t(V),\nabla^V_u)
\tr\Big[\frac{d \nabla_u^V}{du}\frac{t\,e^{\sqrt{-1}R^V_u/(2\pi)}}{1+t\,e^{\sqrt{-1}R^V_u/(2\pi)}}\Big]\Big),
\end{split}
\end{equation}
Let $h^{\Lambda_t(V)}$ be the actions on $\Lambda_t(V)$ induced from the action of $h$ on $V$.
By \cite[(1.25)]{MR1231957} and \eqref{57}, we get that for $t\in \mathbb{C}\backslash\{-1\}$,
\begin{equation}
\begin{split}\label{51}
&\ch(\Lambda_t(V),h^{\Lambda_t(V)},d)
\\
&=-\frac{\sqrt{-1}}{2\pi}\int_{0}^1(1+t)^{\dim V}
\tr\Big[h^{-1}dh \frac{t\,e^{\sqrt{-1}R^V_u/(2\pi)}}{1+t\,e^{\sqrt{-1}R^V_u/(2\pi)}}\Big]du.
\end{split}
\end{equation}

Taking the degree $2j-1$ component of \eqref{51}, we get
\begin{equation}\label{52}
c_{2j-1}\Big(\Lambda_t(V),h^{\Lambda_t(V)},d\Big)=\mathcal{P}_j(t)\,c_{2j-1}(V,h,d), \quad j\geq 1,
\end{equation}
where $\mathcal{P}_j(t)$ is a polynomial in $t$. Now the second item of the proposition follows
by taking the coefficients of $t^m$ of both sides of \eqref{52}.
\end{proof}

Putting $V=TM$ in Theorem \ref{37}, we have

\begin{cor}\label{V=TM}
Suppose $g_*=1$, $c_3(M, [g])=0$. Then the operators
$$\mathcal{T}\otimes \Theta_2(TM)\otimes (E_\mathbb{C},g),
\quad
\mathcal{T}\otimes \Theta_3(TM)\otimes (E_\mathbb{C},g)$$
are rigid.
\end{cor}

\begin{proof}[Proof of Theorem \ref{main}]
Since $M$ is simply connected, $g_*$ is automatically trivial.
Also $c_3(M, [g])\in H^3(M, \mathbb{R})$ is zero. Therefore the conditions of Corollary \ref{V=TM} is verified.
\end{proof}

\begin{cor}\label{V=0}
Assume $M$ is connected and the circle action is nontrivial.
If $g_*=1$, $c_3(M, [g])=0$ and $p_1(TM)_{S^1}=-n\cdot \pi^*u^2$
for some integer $n$, then the equivariant index of the Toeplitz-Witten operator
\begin{equation*}
\mathcal{T}\otimes \bigotimes_{n=1}^\infty S_{q^n}(T_\mathbb{C}M)\otimes (E_\mathbb{C},g)
\end{equation*}
is identically zero.
\end{cor}
\begin{proof}
Taking $V=0$ in the third equality in \eqref{16}, we see that
\begin{equation*}
-\sum_\gamma\sum_j\gamma^2=n,
\end{equation*}
from which we know the $n>0$ case can never happen. If $n=0$, then all the numbers $\dim N_\gamma$ are zero, so that
the fixed point set of the circle action is empty. From Proposition \ref{fixed}, we know
$\mathcal{T}\otimes \bigotimes_{n=1}^\infty S_{q^n}(T_\mathbb{C}M)\otimes (E_\mathbb{C},g)$ has vanishing equivariant index.
For $n<0$, we may take $V=0$ in Theorem \ref{37} to derive the result.
\end{proof}

\subsection{A proof of Theorem \ref{24} }

We continue in the notations of the previous subsection.

Similarly to \eqref{25}, let
\begin{equation}
V\big|_{M^{S^1}}=\bigoplus_{\nu\neq0} V_\nu \oplus V_0^{\mathbb{R}}
\end{equation}
be the $S^1$-equivariant decomposition of the restrictions of $V$ over $M^{S^1}$, where
$V_\nu$ is a complex vector bundle such that $h\in S^1$ acts on it by $h^\nu$ with
$\nu\in \mathbb{Z}\backslash\{0\}$, and $V_0^{\mathbb{R}}$
is the real subbundle of $V\big|_{M^{S^1}}$ such that $S^1$ acts as identity. Set $V_0=V_0^{\mathbb{R}}\otimes\mathbb{C}$.
We denote by
\begin{equation*}
2\pi\sqrt{-1}u_\nu^{\,j},\ j=1,\cdots,\dim V_\nu
\end{equation*}
the Chern roots of $V_\nu$ with $\nu\neq 0$, and by
\begin{equation*}
\pm 2\pi\sqrt{-1}u^{\,j}_0,\ j=1,\cdots, \big[\dim V^\mathbb{R}_0/2\big],
\end{equation*}
the Chern roots of $V_0$.
Here we use the notation that
for $s\in\mathbb{R}$, $[s]$ denotes the greatest integer which is
less than or equal to $s$.

We use the virtual version of the operators in \eqref{38}-\eqref{40}. Set
\begin{align}
\Theta_{1}(TM|V)_v&=\bigotimes_{n=1}^\infty S_{q^n}(\widetilde{T_\mathbb{C}M})\otimes
\bigotimes_{n=1}^{\infty}\Lambda_{q^n}(\widetilde{V_\mathbb{C}}),
\\
\Theta_{2}(TM|V)_v&=\bigotimes_{n=1}^\infty S_{q^n}(\widetilde{T_\mathbb{C}M})\otimes
\bigotimes_{n=1}^{\infty}\Lambda_{-q^{n-{1/2}}}(\widetilde{V_\mathbb{C}}),
\\
\Theta_{3}(TM|V)_v&=\bigotimes_{n=1}^\infty S_{q^n}(\widetilde{T_\mathbb{C}M})\otimes
\bigotimes_{n=1}^{\infty}\Lambda_{q^{n-{1/2}}}(\widetilde{V_\mathbb{C}}).
\end{align}
If $V$ is even dimensional, set
\begin{equation}
\Theta(TM|V)_v=\bigotimes_{n=1}^\infty S_{q^n}(\widetilde{T_\mathbb{C}M})\otimes
\bigotimes_{n=1}^{\infty}\Lambda_{-q^n}(\widetilde{V_\mathbb{C}}).
\end{equation}

We keep the notation explained above Proposition \ref{fixed}, and define some functions on $\mathbb{C}\times \mathbb{H}$,
\begin{align}
F^V_{{\rm L}}(t,\tau)&=-2^{[\dim V/2]}\Big(\frac{-\sqrt{-1}}{2\pi}\Big)^{\dim N}\Big\langle\ch\big(Q_1(E)_v,g^{Q_1(E)_v},d,\tau\big)
\notag
\\
&\hspace{25pt}\cdot\Big(y\frac{\theta'(0,\tau)}{\theta(y,\tau)}\Big)(TM^{S^1})
\cdot\prod_\gamma\Big(\frac{\theta'(0,\tau)}{\theta(x_\gamma+\gamma\,t,\tau)}\Big)(N_\gamma)
\notag
\\
&\hspace{50pt}\cdot\prod_\nu\Big(\frac{\theta_1(u_\nu+\nu\,t,\tau)}{\theta_1(0,\tau)}\Big)(V_\nu),\big[M^{S^1}\big]\Big\rangle,\label{12}
\end{align}
\vspace{-10pt}
\begin{align}
F^V_{{\rm W}}(t,\tau)&=-\Big(\frac{-\sqrt{-1}}{2\pi}\Big)^{\dim N}\Big\langle\ch\big(Q_2(E)_v,g^{Q_2(E)_v},d,\tau\big)
\notag
\\
&\hspace{25pt}\cdot\Big(y\frac{\theta'(0,\tau)}{\theta(y,\tau)}\Big)(TM^{S^1})\notag
\cdot\prod_\gamma\Big(\frac{\theta'(0,\tau)}{\theta(x_\gamma+\gamma\,t,\tau)}\Big)(N_\gamma)
\notag
\\
&\hspace{50pt}\cdot\prod_\nu\Big(\frac{\theta_2(u_\nu+\nu\,t,\tau)}{\theta_2(0,\tau)}\Big)(V_\nu),\big[M^{S^1}\big]\Big\rangle,\label{13}
\end{align}
\vspace{-10pt}
\begin{align}
{F'}^V_{{\rm W}}(t,\tau)&=-\Big(\frac{-\sqrt{-1}}{2\pi}\Big)^{\dim N}\Big\langle\ch\big(Q_3(E)_v,g^{Q_3(E)_v},d,\tau\big)
\notag
\\
&\hspace{25pt}\cdot\Big(y\frac{\theta'(0,\tau)}{\theta(y,\tau)}\Big)(TM^{S^1})
\cdot\prod_\gamma\Big(\frac{\theta'(0,\tau)}{\theta(x_\gamma+\gamma\,t,\tau)}\Big)(N_\gamma)
\notag
\\
&\hspace{6em}
\cdot\prod_v\Big(\frac{\theta_3(u_\nu+\nu\,t,\tau)}{\theta_3(0,\tau)}\Big)(V_\nu),\big[M^{S^1}\big]\Big\rangle,\label{14}
\end{align}
\vspace{-10pt}
\begin{align}
{F}^V_{{\rm dR}\hspace{0.2pt},\hspace{0.3pt}j}(t,\tau)
&=-\,\frac{(-\sqrt{-1})^{\dim N+\,\dim V/2}}{(2\pi)^{\dim N-\,\dim V/2}}
\,\Big\langle\ch\big(Q_j(E)_v,g^{Q_j(E)_v},d,\tau\big)
\notag\\
&\hspace{20pt}\cdot\Big(y\frac{\theta'(0,\tau)}{\theta(y,\tau)}\Big)(TM^{S^1})
\prod_\gamma\Big(\frac{\theta'(0,\tau)}{\theta(x_\gamma+\gamma\,t,\tau)}\Big)(N_\gamma)
\notag\\
&\hspace{40pt}\cdot\prod_\nu\Big(\frac{\theta(u_\nu+\nu\,t,\tau)}{\theta'(0,\tau)}\Big)(V_v),\big[M^{S^1}\big]\Big\rangle.
\quad j=1,2,3.\label{27}
\end{align}

By Proposition \ref{fixed}, we get, for $t\in [0,1]\backslash\mathbb{Q}$ and $h=e^{2\pi\sqrt{-1}\,t}$,
\begin{equation}
F^V_{{\rm L}}(t,\tau)=\ind\Big(h,\mathcal{T}\otimes\Delta(V)\otimes \Theta_1(TM|V)_v\otimes (Q_1(E)_v,g^{Q_1(E)_v})\Big),
\end{equation}
\begin{equation}
F^V_{{\rm W}}(t,\tau)=\ind\Big(h,\mathcal{T}\otimes \Theta_2(TM|V)_v\otimes (Q_2(E)_v,g^{Q_2(E)_v})\Big),
\end{equation}
\begin{equation}
{F'}^V_{{\rm W}}(t,\tau)=\ind\Big(h,\mathcal{T}\otimes \Theta_3(TM|V)_v\otimes (Q_3(E)_v,g^{Q_3(E)_v})\Big),
\end{equation}
\begin{equation}
\begin{split}
{F}^V_{{\rm dR}\hspace{0.2pt},\hspace{0.3pt}j}(t,\tau)&=\ind\Big(h,\mathcal{T}\otimes (\Delta_+(V)-\Delta_-(V))
\\
&\hspace{3em}\otimes \Theta(TM|V)_v\otimes (Q_j(E)_v,g^{Q_j(E)_v})\Big),\quad j=1,2,3.
\end{split}
\end{equation}

Recall that a (meromorphic) Jacobi form of index $m$ and weight $l$ over $L\rtimes\Gamma$, where
$L$ is an integral lattice in the complex plane $\mathbb{C}$ preserved by the modular
subgroup $\Gamma\subset SL_2(\mathbb{Z})$, is a (meromorphic)
function $F(t,\tau)$ over $\mathbb{C}\times\mathbb{H}$ such that
\begin{equation}
\begin{split}\label{42}
F\Big(\frac{t}{c\tau+d},\frac{a\tau+b}{c\tau+d}\Big)
&=(c\tau+d)^le^{2\pi\sqrt{-1}m(ct^2/(c\tau+d))}F(t,\tau),
\\
F(t+\lambda\tau+\mu,\tau)&=e^{-2\pi\sqrt{-1}m(\lambda^2\tau+2\lambda t)}F(t,\tau),
\end{split}
\end{equation}
where $(\lambda,\mu)\in L$, and
$\begin{pmatrix}
a & b
\\
c & d
\end{pmatrix}
\in \Gamma.$
If $F$ is holomorphic over $\mathbb{C}\times\mathbb{H}$, we say that $F$ is a holomorphic Jacobi form.

\

The following theorem can be thought of as an odd analogue of \cite[Theorem 3]{MR1331972} (compare with \cite[Theorem 3.1]{MR2532715}).

\begin{thm}\label{30}
Assume \eqref{22} holds.
\begin{enumerate}[{\rm (i)}]
\item $F^V_{{\rm L}}(t,\tau)$, $F^V_{{\rm W}}(t,\tau)$ and ${F'}^V_{{\rm W}}(t,\tau)$ are holomorphic Jacobi
forms of index ${n}/{2}$ and weight $(\dim M+1)/2$ over $(2\mathbb{Z})^2\rtimes\Gamma$ with
$\Gamma$ equal to $\Gamma_0(2)$, $\Gamma^0(2)$ and $\Gamma_\theta$, respectively.

\item ${F}^V_{{\rm dR},1}(t,\tau)$, ${F}^V_{{\rm dR},2}(t,\tau)$ and ${F}^V_{{\rm dR},3}(t,\tau)$ are holomorphic Jacobi
forms of index ${n}/{2}$ and weight $(\dim M-\dim V+1)/2$ over $(2\mathbb{Z})^2\rtimes\Gamma$ with
$\Gamma$ equal to $\Gamma_0(2)$, $\Gamma^0(2)$ and $\Gamma_\theta$, respectively.
\end{enumerate}
\end{thm}

\

The remaining part is devoted to a proof of Theorem \ref{30}.

First, the condition \eqref{22} implies that
\begin{equation}
\sum_{\nu}\sum_{j}(u_\nu^j+\nu\,t)^2-\sum_{j}y_j^2-\sum_\gamma\sum_j(x_\gamma^j+\gamma\,t)^2=n\,t^2,
\end{equation}
which gives the equalities
\begin{equation}
\begin{split}\label{16}
&\sum_{\nu}\sum_{j}(u_\nu^j)^2-\sum_{j}y_j^2-\sum_\gamma\sum_j(x_\gamma^j)^2=0,
\\
&\sum_{\nu}\sum_{j}\nu u_\nu^j-\sum_\gamma\sum_j\gamma\,x_\gamma^j=0,
\quad
\sum_{\nu}\sum_{j}\nu^2-\sum_\gamma\sum_j\gamma^2=n.
\end{split}
\end{equation}
Under the action $t\longmapsto t+\lambda\tau+\mu$ with $\lambda,\mu \in 2\mathbb{Z}$, we have
\begin{equation}
\begin{split}\label{41}
\theta(x_\gamma+\gamma(t+\lambda\tau+\mu))
&=e^{-\pi \sqrt{-1}(\gamma^2(\lambda^2\tau+2\lambda t)+2\gamma \lambda x_\gamma)}\theta(x+\gamma\,t,\tau),
\\
\theta_1(u_\nu+\nu(t+\lambda\tau+\mu))
&=e^{-\pi \sqrt{-1}(\nu^2(\lambda^2\tau+2\lambda t)+2 \nu\lambda u_\nu)}\theta_1(u_\nu+\nu\,t,\tau).
\end{split}
\end{equation}
From \eqref{16} and \eqref{41}, we see that $F^V_{{\rm L}}(t,\tau)$ verifies the second line of \eqref{42} with $m=\frac{n}{2}$.
In a very similar way, we can show that for
\begin{equation*}
F\in \{F^V_{{\rm W}},{F'}^V_{{\rm W}},{F}^V_{{\rm dR},1},{F}^V_{{\rm dR},2},{F}^V_{{\rm dR},3}\},
\end{equation*}
$F$ also verifies the second line of \eqref{42} with index $m=\frac{n}{2}$.

Similar to \cite[Lemma 3.2]{MR1331972}, we have the following transformation formulas
(compare with \cite[Lemma 3.2]{MR2532715}).

\begin{lem}\label{29}
\begin{enumerate}[{\rm (i)}]
\item The following equalities hold,
\begin{equation}\label{21}
F^V_{{\rm L}}(t,\tau+1)=F^V_{{\rm L}}(t,\tau), \ F^V_{{\rm W}}(t,\tau+1)={F'}^V_{{\rm W}}(t,\tau),
\end{equation}
\begin{equation}\label{26}
{F}^V_{{\rm dR},1}(t,\tau+1)={F}^V_{{\rm dR},1}(t,\tau),\ {F}^V_{{\rm dR},2}(t,\tau+1)={F}^V_{{\rm dR},3}(t,\tau).
\end{equation}

\item Assume \eqref{22} holds. Then
\begin{equation}
\begin{split}
F^V_{{\rm L}}\Big(\frac{t}{\tau},-\frac{1}{\tau}\Big)
&=2^{[(N+\dim V)/2]}\tau^{(\dim M +1)/2}e^{\pi\sqrt{-1}nt^2/\tau}F^V_{{\rm W}}(t,\tau),\label{18}
\\
{F'}^V_{{\rm W}}\Big(\frac{t}{\tau},-\frac{1}{\tau}\Big)
&=\tau^{(\dim M +1)/2}e^{\pi\sqrt{-1}nt^2/\tau}{F'}^V_{{\rm W}}(t,\tau),
\end{split}
\end{equation}
\begin{equation}
\begin{split}
{F}^V_{{\rm dR},1}\Big(\frac{t}{\tau},-\frac{1}{\tau}\Big)
&=2^{N/2}\tau^{(\dim M -\dim V+1)/2}e^{\pi\sqrt{-1}nt^2/\tau}{F}^V_{{\rm dR},2}(t,\tau),\label{28}
\\
{F}^V_{{\rm dR},3}\Big(\frac{t}{\tau},-\frac{1}{\tau}\Big)
&=\tau^{(\dim M -\dim V+1)/2}e^{\pi\sqrt{-1}nt^2/\tau}{F}^V_{{\rm dR},3}(t,\tau).
\end{split}
\end{equation}
\end{enumerate}
\end{lem}

\begin{proof}
\eqref{21} and \eqref{26} follow from \cite[(3.28)-(3.31), (4.6)]{MR2495834}, \eqref{1} and \eqref{12}-\eqref{27} straightforwardly.

From \cite[(3.28)-(3.31), (4.7)]{MR2495834}, we obtain the following transformation formulas,
\begin{equation}
\begin{split}\label{17}
y\frac{\theta'(0,-1/\tau)}{\theta(y,-1/\tau)}
&=e^{-\pi\sqrt{-1}\tau y^2}\tau y\frac{\theta'(0,\tau)}{\theta(\tau y,\tau)},
\\
\frac{\theta'(0,-1/\tau)}{\theta(x_\gamma+\gamma\,t/\tau,-1/\tau)}
&=e^{-\pi\sqrt{-1}\tau(x_\gamma+\gamma\,t/\tau)^2}\,
\frac{\tau\theta'(0,\tau)}{\theta(\tau x_\gamma+\gamma\,t,\tau)}\ ,
\\
\frac{\theta_1(u_\nu+\nu\,t/\tau,-1/\tau)}{\theta_1(0,-1/\tau)}
&=e^{\pi\sqrt{-1}\tau(u_\nu+\nu\,t/\tau)^2}\,
\frac{\theta_2(\tau u_\nu+\nu\,t,\tau)}{\theta_2(0,\tau)}\ .
\end{split}
\end{equation}

Combining \eqref{64}, \eqref{12}, \eqref{16} and \eqref{17}, we obtain the first line in \eqref{18}.
The other lines in \eqref{18} and \eqref{28} can be verified in a similar way.
\end{proof}

Since the generators of $\Gamma_0(2)$ are $T$, $ST^2ST$, from Lemma \ref{29}, we can check directly that
the first line of \eqref{42} holds for $F^V_{{\rm L}}(t,\tau)$ with
$\Gamma=\Gamma_0(2)$, $m=\frac{n}{2}$, $l=\frac{\dim M+1}{2}$.
Thus, $F^V_{{\rm L}}(t,\tau)$ is a meromorphic Jacobi form
of index $\frac{n}{2}$ and weight $\frac{\dim M+1}{2}$ over $(2\mathbb{Z})^2\rtimes\Gamma_0(2)$.
In a similar way, we can prove the assertions in Theorem \ref{30} except the holomorphic property.

For $\mathscr{G}=
\begin{pmatrix}
a & b
\\
c & d
\end{pmatrix}
\in \SL_2(\mathbb{Z})$ and a Jacobi form $F$ of index $m$ and weight $l$, we write
\begin{equation}\label{31}
F\big(\mathscr{G}(t,\tau)\big)
=(c\tau+d)^{-l}e^{-2\pi\sqrt{-1}m(ct^2/(c\tau+d))}F\Big(\frac{t}{c\tau+d},\frac{a\tau+b}{c\tau+d}\Big).
\end{equation}

Lemma \ref{29} tells us that, if
\begin{equation*}
F\in \{F^V_{{\rm L}},F^V_{{\rm W}},{F'}^V_{{\rm W}}\}\ \big(\text{resp. }
\{{F}^V_{{\rm dR},1},{F}^V_{{\rm dR},2},{F}^V_{{\rm dR},3}\}\big),
\end{equation*}
its modular transformation $F\big(\mathscr{G}(t,\tau)\big)$
which may need a multiplication by a constant is still in
\begin{equation*}
\{F^V_{{\rm L}},F^V_{{\rm W}},{F'}^V_{{\rm W}}\}\ \big(\text{resp. }
\{{F}^V_{{\rm dR},1},{F}^V_{{\rm dR},2},{F}^V_{{\rm dR},3}\}\big).
\end{equation*}

\begin{lem}\label{32}
For any function
\begin{equation*}
F\in \{F^V_{{\rm L}},F^V_{{\rm W}},{F'}^V_{{\rm W}},{F}^V_{{\rm dR},1},{F}^V_{{\rm dR},2},{F}^V_{{\rm dR},3}\},
\end{equation*}
its modular transformation
is holomorphic in $(t,\tau)\in \mathbb{R}\times\mathbb{H}$.
\end{lem}
\begin{proof}
The proof is almost the same as the proof of \cite[Lemma 2.3]{MR1331972} except that
we use Proposition \ref{fixed} instead of the Lefschetz fixed point formula therein.
\end{proof}

It is crucial that $F$ and its modular transformation are the Lefschetz
numbers of certain Toeplitz operators. This is also the place where the spin
conditions on $M$ and $V$ as well as the assumptions on $g$ come in.
Thus, one can use index theory to cancel part of the poles of these functions.

We now prove that
\begin{equation*}
F\in \{F^V_{{\rm L}},F^V_{{\rm W}},{F'}^V_{{\rm W}},{F}^V_{{\rm dR},1},{F}^V_{{\rm dR},2},{F}^V_{{\rm dR},3}\}
\end{equation*}
is actually holomorphic on $\mathbb{C}\times \mathbb{H}$.

The proof essentially makes use of Liu's key techniques \cite[Lemma 3.4]{MR1331972}.
We give the details here for completeness.
By \eqref{12} and \eqref{27}, we see that the possible poles of $F(t,\tau)$ can be
written in the form $t=k(c\tau+d)/r$ for integers $k$, $r$, $c$, $d$ with $(c,d)=1$.

Suppose $t=k(c\tau+d)/r$ is a pole for $F(t,\tau)$. Find integers $a$, $b$ such that $ad-bc=1$.
Take
$\mathscr{G}=
\begin{pmatrix}
d & -b
\\
-c & a
\end{pmatrix}
\in \SL_2(\mathbb{Z})$.
From \eqref{31}, it is easy to see that $F\big(\mathscr{G}(t,\tau)\big)$
and $F\Big(\frac{t}{-c\tau+a},\frac{d\tau-b}{-c\tau+a}\Big)$ have the same poles.
Now that $t=k(c\tau+d)/r$ is a pole of $F(t,\tau)$, a pole of $F\big(\mathscr{G}(t,\tau)\big)$
is given by solving the equation
\begin{equation}
\frac{t}{-c\tau+a}=\frac{k(c\frac{d\tau-b}{-c\tau+a}+d)}{r},
\end{equation}
which exactly gives $t=k/r\in\mathbb{R}$. By Lemma \ref{32}, we get a contradiction.
Therefore, $F(t,\tau)$ is holomorphic on $\mathbb{C}\times \mathbb{H}$.

We complete the proof of Theorem \ref{30}.


\end{document}